\documentclass[graybox]{svmult}

\usepackage{mathptmx}       
\usepackage{helvet}         
\usepackage{courier}        
\usepackage{type1cm}        
\usepackage{makeidx}         
\usepackage{graphicx}        
\usepackage{multicol}        

\usepackage{amssymb,amsfonts,amscd,amsmath,amscd}

\input amssym.def

\input amssym.tex

\DeclareFontFamily{OT1}{pzc}{}
\DeclareFontShape{OT1}{pzc}{m}{it}%
             {<-> s * [0.900] pzcmi7t}{}
\DeclareMathAlphabet{\mathscr}{OT1}{pzc}%
                                 {m}{it}

\newcommand{\ce}{\mathcal{E}}

\newcommand{\pp}{\ensuremath{\mathbb{P}}}
\newcommand\be{\begin{equation}}
\newcommand\ee{\end{equation}}
\newcommand\bea{\begin{eqnarray}}
\newcommand\eea{\end{eqnarray}}
\newcommand\bi{\begin{itemize}}
\newcommand\ei{\end{itemize}}
\newcommand\ben{\begin{enumerate}}
\newcommand\een{\end{enumerate}}

\newcommand{\twocase}[5]{#1 \begin{cases} #2 & \text{{\rm #3}}\\ #4 &\text{{\rm #5}} \end{cases}   }
\newcommand{\ncr}[2]{{#1 \choose #2}}

\newtheorem{thm}{Theorem}[section]

\newtheorem{lem}[thm]{Lemma}

\newtheorem{defi}[thm]{Definition}

\numberwithin{subsubsection}{subsection}

\begin{document}

\title{Gaussian Behavior in Generalized Zeckendorf Decompositions}

\author{Steven J. Miller and Yinghui Wang}
\institute{Department of Mathematics and Statistics, Williams College, sjm1@williams.edu (Steven.Miller.MC.96@aya.yale.edu) and Department of Mathematics, MIT, yinghui@mit.edu}

%
%
\maketitle

\abstract{A beautiful theorem of Zeckendorf states that every integer can be written uniquely as a sum of non-consecutive Fibonacci numbers $\{F_n\}_{n=1}^{\infty}$; Lekkerkerker proved that the average number of summands for integers in $[F_n, F_{n+1})$ is $n/(\varphi^2 + 1)$, with $\varphi$ the golden mean. Interestingly, the higher moments seem to have been ignored. We discuss the proof that the distribution of the number of summands converges to a Gaussian as $n \to \infty$, and comment on generalizations to related decompositions. For example, every integer can be written uniquely as a sum of the $\pm F_n$'s, such that every two terms of the same (opposite) sign differ in index by at least 4 (3). The distribution of the numbers of positive and negative summands converges to a bivariate normal with computable, negative correlation, namely $-(21-2\varphi)/(29+2\varphi) \approx -0.551058$.}




\section{Introduction}

\subsection{History}\label{history}
$ $

The Fibonacci numbers have fascinated professional mathematicians and amateurs for centuries. They have a wealth of properties and interesting relationships; see for example \cite{Kos}. This article is concerned with how Fibonacci numbers arise in decompositions. A beautiful theorem of Zeckendorf \cite{Ze} states that every positive integer may be written uniquely as a sum of non-adjacent Fibonacci numbers; of course, to ensure that the decomposition is unique we need to use the normalization $F_1 = 1$, $F_2 = 2$, $F_3 = 3$, $F_4 = 5$ and in general $F_{n+1} = F_n + F_{n-1}$.

The standard proof is by induction, and is also constructive. Simply take the largest Fibonacci number that is at most our number $x$; say this is $F_n$. If $F_n+F_{n-1} \le x$ then we could take $F_{n-1}$ in our decomposition, and then replace $F_n$ and $F_{n-1}$ with $F_{n+1}$ by the recurrence relation, which contradicts the maximality of $F_n$. Thus $F_n + F_{n-1} > x$ and $F_{n-1}$ cannot be used. By induction, we can write $x-F_n$ as a sum of non-adjacent Fibonacci numbers; further, since $x-F_n < F_{n-1}$ clearly $F_{n-1}$ is not in our decomposition. Adding $F_n$ to the decomposition of $x-F_n$ yields our desired expansion. To prove uniqueness, we again proceed by induction. Arranging the Fibonacci summands in decreasing order, if $x = F_{n_1} + \cdots + F_{n_k} = F_{m_1} + \cdots + F_{m_\ell}$, we must have $F_{n_1} = F_{m_1}$; if not, the two sums cannot be equal.\footnote{If $F_{m_1} > F_{n_1}$, then the largest the $n$-sum can be is $F_{n_1} + F_{n_1-2} + F_{n_1-4} + \cdots + \delta$, where $\delta$ is either 1 or 2. Adding 1 or 2 to this and using the recurrence relation gives $F_{n_1+1}$, and thus the expressions are unequal.} The claim now follows by induction as $x-F_{n_1} < x$.

There are many questions one may ask about the Zeckendorf decomposition. The first is to understand the average number of summands needed. Lekkerkerker \cite{Lek} proved that for $x \in [F_n, F_{n+1})$, as $n\to\infty$ the average number of summands needed is $n/(\varphi^2 + 1)$, with $\varphi = \frac{1+\sqrt{5}}2$ the golden mean. The proof involves continued fractions, and is unfortunately limited to the mean (though it can be generalized to other related decompositions; see \cite{BCCSW,Day,Ha,Ho,Ke,Len} for some of the history and results along these lines).

In this chapter we discuss the combinatorial vantage of Kolo$\breve{{\rm g}}$lu, Kopp, Miller and Wang that allows us to determine, not just the mean, but all the moments and thus the limiting distribution. They prove the fluctuations of the number of summands needed for $x \in [F_n,F_{n+1})$  about the mean converges to a Gaussian as $n\to\infty$. The proof follows from writing down an explicit formula for the number of $x$ where there are exactly $k$ summands in the Zeckendorf decomposition, and then using Stirling's formula to analyze the resulting binomial coefficients. These results were generalized in Miller-Wang \cite{MW} to other decompositions in terms of elements satisfying special recurrence properties. While the combinatorial approach is still the foundation of the analysis, the resulting expressions are too involved and cannot be attacked as directly as the Stirling approach in the initial case. Instead, the proofs are completed by an analysis of the resulting generating functions, and then differentiating identities to determine the moments.

We describe the main results in \S\ref{mainresults} and then sketch the proofs in the Fibonacci case in \S\ref{sec:fibonaccicase}. The key idea is recasting the problem from number theory to combinatorics, specifically to the number of integer partitions satisfying certain constraints. We describe this idea in detail, and outline the resulting algebra (complete details are available in \cite{KKMW}). Unfortunately, while the combinatorial framework applies to generalizations of Zeckendorf decompositions, the proof in general is significantly harder. This is due to the fact that for the Fibonacci numbers, we have an explicit formulas for the probability a number in $[F_n, F_{n+1})$ has exactly $k+1$ summands in its Zeckendorf decomposition; it is just $\ncr{n-1-k}{k}/F_{n-1}$. All our results follow from a careful analysis of the behavior as $n\to\infty$, which can be accomplished with Stirling's formula. In the general case, the resulting expressions are not as tractable, and the combinatorial approach must be supplemented. We sketch the main ingredients in the analysis in \S\ref{sec:plrs}, and refer the reader to \cite{MW} for full proofs. We end in \S\ref{sec:conclusion} with some open questions.

\subsection{Main Results}\label{mainresults}
$ $

Before stating the main results, we first set some notation. The sequences defined below are the generalizations of the Fibonacci numbers that can be handled using the techniques of \cite{MW}.

\begin{defi}\label{defn:goodrecurrencereldef}\label{def:goodrecurrence} A sequence $\{H_n\}_{n=1}^\infty$ of positive integers is a \textbf{Positive Linear Recurrence Sequence (PLRS)} if the following properties hold:

\ben
\item \emph{Recurrence relation:} There are non-negative integers $L, c_1, \dots, c_L$\label{c_i} such that \be H_{n+1} \ = \ c_1 H_n + \cdots + c_L H_{n+1-L},\ee with $L, c_1$ and $c_L$ positive.
\item \emph{Initial conditions:} $H_1 = 1$, and for $1 \le n < L$ we have
\be H_{n+1} \ =\
c_1 H_n + c_2 H_{n-1} + \cdots + c_n H_{1}+1.\ee
\een

A decomposition $\sum_{i=1}^{m} {a_i H_{m+1-i}}$\label{a_i} of a positive integer $N$ (and the sequence $\{a_i\}_{i=1}^{m}$) is \textbf{legal}\label{legal} if $a_1>0$, the other $a_i \ge 0$, and one of the following two conditions holds:

\begin{itemize}

\item Condition 1: 
We have $m<L$ and $a_i=c_i$ for $1\le i\le m$.

\item Condition 2: There exists $s\in\{0,\dots, L\}$ such that
\begin{equation}\label{eq:legalcondition2}
a_1\ = \ c_1,\ a_2\ = \ c_2,\ \cdots,\ a_{s-1}\ = \ c_{s-1}\ {\rm{and}}\ a_s<c_s,
\end{equation}
$a_{s+1}, \dots, a_{s+\ell} \ = \  0$ for some $\ell \ge 0$,
and $\{b_i\}_{i=1}^{m-s-\ell}$ (with $b_i = a_{s+\ell+i}$) is legal.

\end{itemize}

If $\sum_{i=1}^{m} {a_i H_{m+1-i}}$ is a legal decomposition of $N$, we define the \textbf{number of summands}\label{summands} (of this decomposition of $N$) to be $a_1 + \cdots + a_m$.
\end{defi}

Informally, a legal decomposition is one where we cannot use the recurrence relation to replace a linear combination of summands with another summand, and the coefficient of each summand is appropriately bounded. For example, if $H_{n+1} = 2 H_n + 3 H_{n-1} + H_{n-2}$, then $H_5 + 2 H_4 + 3 H_3 + H_1$ is legal, while $H_5 + 2 H_4 + 3 H_3 + H_2$ is not (we can replace $2 H_4 + 3 H_3 + H_2$ with $H_5$), nor is $7H_5 + 2H_2$ (as the coefficient of $H_5$ is too large). Note the Fibonacci numbers are just the special case of $L=2$ and $c_1 = c_2 = 1$.

The following probabilistic language is convenient for stating some of our main results.

\begin{defi}[Associated Probability Space to a Positive Linear Recurrence Sequence]\label{def:assocprobspace}
Let $\{H_n\}$ be a Positive Linear Recurrence Sequence. For each $n$, consider the discrete outcome space \be \Omega_n \ = \ \{H_n,\ H_n+1,\ H_n + 2,\  \cdots,\ H_{n+1}-1\} \ee with probability measure \be \pp_n(A) \ = \ \sum_{\omega \in A \atop \omega \in \Omega_n} \frac1{H_{n+1}-H_n}, \ \ \  a \subset \Omega_n; \ee in other words, each of the $H_{n+1}-H_n$ numbers is weighted equally. We define the random variable $K_n$\label{K_n} by setting $K_n(\omega)$ equal to the number of summands of $\omega \in \Omega_n$ in its legal decomposition. Implicit in this definition is that each integer has a unique legal decomposition; we prove this in Theorem \ref{thm:genZeckendorf}, and thus $K_n$ is well-defined.

We denote the cardinality of $\Omega_n$\label{index:Deltan}  by \be \Delta_n\ =\ H_{n+1}-H_n,\ee and we set $p_{n,k}$\label{pnk} equal to the number of elements in $[H_n, H_{n+1})$ whose generalized Zeckendorf decomposition has exactly $k$ summands; thus \be p_{n,k}\ =\ \Delta_n \cdot {\rm Prob}(K_n=k).\ee
\end{defi}

Before stating the main results, we first examine a special case which suggests why the limiting behavior is Gaussian. Consider the PLRS given by $L=1$, $H_{n+1} = 10 H_n$. Thus our PLRS is just the geometric series $1, 10, 100, \dots$, and a legal decomposition of $N$ is just its decimal expansion. Clearly every positive integer has a unique legal decomposition. Further, the distribution of the number of summands converges to a Gaussian by the Central Limit Theorem, as we essentially have the sum of $n-1$ independent, identically distributed discrete uniform random variables. To see this, write $N = a_1 10^n + \cdots + a_{n+1} 1$. We are interested in the large $n$ behavior of $a_1 + \cdots + a_{n+1}$ as we vary over $x$ in $[10^n, 10^{n+1})$. For large $n$ the contribution of $a_1$ is immaterial, and the remaining $a_i$'s can be understood by considering the sum of $n$ independent, identically distributed discrete uniform random variables on $\{0, \dots, 9\}$ (which have mean 4.5 and standard deviation approximately 2.87). Denoting these random variables by $A_i$\label{A_i}, by the Central Limit Theorem $A_2 + \cdots + A_{n+1}$ converges to being normally distributed with mean $4.5 n$ and standard deviation $\sqrt{33n/4}$.

The first result is that a PLRS leads to a unique, generalized Zeckendorf decomposition. Part (a) has also been recently studied by Hamlin \cite{Ha}.

\begin{theorem}[Generalized Zeckendorf's Theorem for PLRS \cite{MW}]\label{thm:genZeckendorf} Let $\{H_n\}_{n=1}^\infty$ be a \emph{Positive Linear Recurrence Sequence}. Then

{\rm{(a)}} There is a unique legal decomposition for each integer $N\ge 0$.

{\rm{(b)}} There is a bijection between the set $\mathcal{S}_n$\label{mathcalS_n} of integers in $[H_n, H_{n+1})$ and the set $\mathcal{D}_n$\label{mathcalD_n} of legal decompositions $\sum_{i=1}^{n} {a_i H_{n+1-i}}$.
\end{theorem}

In addition to being of interest in its own right, the Generalized Lekkerkerker Theorem below is needed as input to prove the Gausian behavior of the number of summands.

\begin{theorem}[Generalized Lekkerkerker's Theorem for PLRS \cite{MW}]\label{thm:genlekkerkerker} Let $\{H_n\}_{n=1}^\infty$ be a \emph{Positive Linear Recurrence Sequence}, let $K_n$ be the random variable of Definition \ref{def:assocprobspace} and denote its mean by $\mu_n$. Then there exist constants $C>0$, $d$ and $\gamma_1\in (0,1)$ depending only on $L$ and the $c_i$'s in the recurrence relation of the $H_n$'s such that
\begin{equation}\label{eq:genlekkerkerker}
\mu_n\ =\ Cn+d+o(\gamma_1^n).
\end{equation}
\end{theorem}

The main result in the subject is

\begin{theorem}[Gaussian Behavior for PLRS \cite{MW}]\label{thm:Gaussian} Let $\{H_n\}_{n=1}^\infty$ be a \emph{Positive Linear Recurrence Sequence} and let $K_n$ be the random variable of Definition \ref{def:assocprobspace}. As $n\rightarrow \infty$, the distribution of $K_n$ converges to a Gaussian.
\end{theorem}

Their method generalizes to a multitude of other problems, and allows us to prove Gaussian behavior in many other situations. We state one particularly interesting situation.

\begin{definition}\label{far}
We call a sum of the $\pm F_n$'s a {\textit{\textbf{far-difference representation}}}\label{far-difference representation} if every two terms of the same sign differ in index by at least 4, and every two terms of opposite sign differ in index by at least 3.
\end{definition}

Albert \cite{Al} proved the analogue of Zeckendorf's Theorem for the far-difference representation. It is convenient to set\label{S_n} \be \twocase{S_n \ = \ }{\sum_{0<n-4i\le n} F_{n-4i} \ = \ F_n + F_{n-4} + F_{n-8} + \cdots}{if $n > 0$}{0}{otherwise.} \ee

\begin{theorem}[Generalized Zeckendorf's Theorem for Far-Difference Representations \cite{Al}]\label{thmhan}
Every integer has a unique far-difference representation. For each $N\in (S_{n-1}=F_n-S_{n-3}-1,S_n]$, the first term in its far-difference representation is $F_n$, and the unique far-difference representation of 0 is the empty representation.
\end{theorem}

The far-difference representations have both positive and negative summands, which opens up the fascinating question of how the number of each are related. It turns out they are correlated, though the total number of summands and the difference in the number of positive and negative summands are not. Specifically,

\begin{theorem}[Generalized Lekkerkerker's Theorem and Gaussian Behavior for Far-Difference Representations \cite{MW}]\label{thm:lekgaussfardiff}
Let $\mathcal{K}_n$\label{mathcalK_n} and $\mathcal{L}_n$ be the corresponding random variables denoting the number of positive summands and the number of negative summands in the far-difference representation for integers in $(S_{n-1},S_n]$. As $n$ goes to infinity, the expected value of $\mathcal{K}_n$, denoted by $\mathbb{E}[\mathcal{K}_n]$\label{mathbbE}, is $\frac{1}{10}n+\frac{371-113\sqrt{5}}{40}$ and $\frac{1+\sqrt{5}}{4}=\frac{\phi}{2}$ greater than $\mathbb{E}[\mathcal{L}_n]$; the variance of both is of size $\frac{15+21\sqrt{5}}{1000}n$; the joint density of $\mathcal{K}_n$ and $\mathcal{L}_n$ is a bivariate Gaussian with negative correlation $\frac{10\sqrt{5}-121}{179}=-\frac{21-2\varphi}{29+2\varphi}\approx -0.551$; and $\mathcal{K}_n+\mathcal{L}_n$ and $\mathcal{K}_n-\mathcal{L}_n$ are independent.
\end{theorem}

\ \\

\noindent \emph{Acknowledgements:} The first named author was partially supported by NSF grant DMS0970067 and the second named author was partially supported by NSF grant DMS0850577, Williams College and the MIT Mathematics Department. It is a pleasure to thank our colleagues from the Williams College 2010 SMALL REU program for many helpful conversations, especially Ed Burger, David Clyde, Cory Colbert, Carlos Dominguez, Gene Kopp, Murat Kolo$\breve{{\rm g}}$lu, Gea Shin and Nancy Wang. The first named author also thanks Cameron and Kayla Miller for discussions related to the cookie problem, Ed Scheinerman for useful comments on an earlier draft, as well as participants from CANT 2009 and 2010, where versions of this work was presented.


\section{The Fibonacci Case}\label{sec:fibonaccicase}

As the framework of the Fibonacci case is the basis for the general case, and as the Fibonacci case can be handled elementarily, we describe it in some detail. The key idea is to change our perspective and apply the solution to a standard problem in combinatorics, the stars and bars problem. It is also known as the cookie problem in some circles (see for instance Chapter 1 of \cite{MT-B}), or Waring's problem with first powers (see \cite{Na}).

\begin{lem}\label{lem:cookieproblem} The number of ways to partition $C$ identical objects into $P$ distinct sets is $\ncr{C+P-1}{P-1}$. \end{lem}

Here it is very important that the objects being partitioned are indistinguishable; all that matters is how  many objects are given to a specified set, not which objects are given.

\begin{proof} Imagine instead $C+P-1$ objects. There are $\ncr{C+P-1}{P-1}$ ways to choose $P-1$ of the $C$. Each of these choices corresponds to a partition of $C$ objects into $P$ sets where order does not count and the objects are indistinguishable. Specifically, for a given choice all of the remaining $C$ items are divided among the $P$ sets, where everything up to the first of the $P-1$ chosen objects goes to the first set, the objects between the first and second chosen element goes into the second set, and so on.
\end{proof}

We may recast the above as saying the number of solutions to $x_1 + \cdots + x_P = C$, with each $x_i \ge 0$, is $\ncr{C+P-1}{P-1}$. One of the advantages of this approach is that it is very easy to add in lower restrictions. For example, to find the number of solutions to $y_1 + \cdots + y_P = C$ with $y_i \ge n_i$ for some choice of non-negative integers (with of course $n_1 + \cdots + n_P \le C$), simply let $y_i = x_i+n_i$. Now each $x_i \ge 0$ is integer valued, and the number of solutions to the equation in the $y_i$'s is the number of solutions to $x_1 + \cdots + x_P = C - (n_1 + \cdots + n_P)$, or $\ncr{C - (n_1 + \cdots + n_P) + P-1}{P-1}$.

We use the above perspective to analyze the Zeckendorf decomposition of an $x \in [F_n, F_{n+1})$. The number of such $x$ is $F_{n+1} - F_n$, which by the Fibonacci recurrence relation is just $F_{n-1}$. We count the number of $x \in [F_n, F_{n+1})$ with exactly $k+1$ summands in their Zeckendorf decomposition, denoting this quantity by $N_n(k)$; we choose to record the number of summands as $k+1$ as each $x$ must have at least one summand, namely $F_n$. By the standard proof of Zeckendorf's theorem, we know each $x$ has at most one such valid decomposition. The existence of a Zeckendorf decomposition follows by showing $\sum_{k=0}^n P_n(k) = F_{n-1}$ (since no number has two valid decompositions, it is enough to know that the number of valid decompositions equals the number of integers in the interval $[F_n, F_{n+1})$).

\begin{lem} Let $P_n(k) = N_n(k) / F_{n-1}$, which is the probability an $x \in [F_n, F_{n+1})$ has exactly $k+1$ summands in its Zeckendorf decomposition. Then 
$P_n(k) = \ncr{n-1-k}{k} / F_{n-1}$. \end{lem}

\begin{proof} If $x$ has exactly $k+1$ summands, then $x = F_{i_1} + F_{i_2} + \cdots + F_{i_k} + F_{i_{k+1}}$, where $F_{i_{k+1}} = F_n$, $1 \le i_1 < i_2 < i_3 < \cdots < i_k < i_{k+1}$, and $d_j := i_{j} - i_{j-1} \ge 2$ for $2 \le j \le k+1$ (and $d_1 := i_1 - 0 \ge 1$). We can recast this in terms of the cookie problem above. Clearly $d_1 + d_2 + \cdots + d_{k+1} = n$ (we have a telescoping series and the last summand in the decomposition is $F_n$). Let $d_1 = x_1 + 1$ and $d_j = x_j + 2$ for $2 \le j \le k+1$. Then the number of $x$ that have exactly $k+1$ summands is the number of tuples $(d_1,\dots,d_{k+1})$ with $d_1 + d_2  + \cdots + d_{k+1} = n$, or equivalently the number of tuples $(x_1, \dots, x_{k+1})$ with $x_1+ x_2 + \cdots + x_{k+1} = n - (2k+1)$. By our combinatorial result, Lemma \ref{lem:cookieproblem}, taking $C = n - (2k+1)$ and $P=k+1$ we see the number of such tuples is just \be N_n(k) \ = \ \ncr{C+P-1}{P-1}\ =\ \ncr{n-(2k+1)+(k+1-1)}{k+1-1}\ =\ \ncr{n-1-k}{k};\ee the lemma now follows.
\end{proof}

Now that we have an explicit formula for $N_n(k)$ and $P_n(k)$, \emph{all} the claims follow for the Fibonacci case. We quickly provide a sketch; see \cite{KKMW} for details.\\

\begin{enumerate}

\item Zeckendorf expansion (Theorem \ref{thm:genZeckendorf}): As each $x \in [F_n, F_{n+1})$ has at most one Zeckendorf decomposition, the claim follows by counting the number of valid Zeckendorf expansions and seeing that this equals $F_{n-1}$. This number is just $\sum_{k=0}^{n} \ncr{n-1-k}{k}$. The summands vanish if $k \ge \lfloor (n-1)/2\rfloor$ (the binomial coefficients are extended so that $\ncr{n}{\ell} = 0$ if $\ell > n$ or $\ell < 0$). We claim \be\label{eq:sumpnkFnminus1}\sum_{k=0}^n \ncr{n-1-k}{k} \ = \ F_{n-1}.\ee We proceed by induction. The base case is clear, and the general case follows from using the standard identity that $\ncr{m}{\ell} + \ncr{m}{\ell+1} = \ncr{m+1}{\ell+1}$. Specifically, \bea \sum_{k=0}^n \ncr{n+1-1-k}{k} & \ = \ & \sum_{k=0}^n \left[ \ncr{n-1-k}{k-1} + \ncr{n-1-k}{k} \right] \nonumber\\ &=& \sum_{k=1}^n \ncr{n-2-(k-1)}{k-1} + \sum_{k=0}^n \ncr{n-1-k}{k} \nonumber\\ &=& \sum_{k=0}^{n-1} \ncr{n-2-k}{k} + \sum_{k=0}^n \ncr{n-1-k}{k} \nonumber\\ &=&
F_{n-2} + F_{n-1} \eea by the inductive assumption; noting
$F_{n-2}+F_{n-1} = F_n$ completes the proof.  \\

\item Lekkerkerker's Theorem (Theorem \ref{thm:genlekkerkerker}): The claim follows by computing the expected value, which is \be \sum_{k=0}^n (k+1) P_n(k) \ = \ 1+\frac1{F_{n-1}} \sum_{k=0}^n k \ncr{n-1-k}{k}.\ee Let \bea
\ce(n) & \ = \ & \sum_{k=0}^{\lfloor \frac{n-1}2\rfloor} k
\ncr{n-1-k}{k}.  \eea Straightforward algebra shows \be\label{eq:approxcen} \ce(n) \ = \  (n-2)
F_{n-3} - \ce(n-2).\ee To see this, note \bea
\ce(n) & \ = \ & \sum_{k=0}^{\lfloor \frac{n-1}2\rfloor} k
\ncr{n-1-k}{k} \ = \ \sum_{k=1}^{\lfloor
\frac{n-1}2\rfloor} k \frac{(n-1-k)!}{k!(n-1-2k)!}
\nonumber\\ &=&  \sum_{k=1}^{\lfloor
\frac{n-1}2\rfloor} (n-2-(k-1))
\frac{(n-3-(k-1)!}{(k-1)!(n-3-2(k-1))!} \nonumber\\ &=&
\sum_{\ell=0}^{\lfloor \frac{n-3}2\rfloor} (n-2-\ell)
\ncr{n-3-\ell}{\ell} \nonumber\\ &=& (n-2) \sum_{\ell=0}^{\lfloor
\frac{n-3}2\rfloor} \ncr{n-3-\ell}{\ell} - \sum_{\ell=0}^{\lfloor
\frac{n-3}2\rfloor} \ell \ncr{n-3-\ell}{\ell} \nonumber\\ &=& (n-2)
F_{n-3} - \ce(n-2), \eea which proves the claim. Note that we
used \eqref{eq:sumpnkFnminus1} to
replace the sum of binomial coefficients with a Fibonacci number.

We study the telescoping sum \bea \sum_{\ell = 0}^{\lfloor \frac{n-3}{2}\rfloor} (-1)^\ell \left(\ce(n-2\ell) + \ce(n-2(\ell+1))\right)  \eea (which is essentially $\ce(n)$). Using \eqref{eq:approxcen} yields
\bea \sum_{\ell=0}^{\lfloor \frac{n-3}{2}\rfloor}
(-1)^\ell (n-3-2\ell) F_{n-3-2\ell} + O(F_{n-2}). \eea While we could
evaluate the last sum exactly, trivially estimating it suffices to
obtain the main term, which will give us Lekkerkerker's Theorem. 

Binet's formula\footnote{It is worth noting that while one can establish Binet's formula by substituting and checking , it can also be derived via generating functions; it is no coincidence that in the general proof generating functions will play a central role.} states that $F_n = \frac{\varphi}{\sqrt{5}} \cdot \varphi^n - \frac{1-\varphi}{\sqrt{5}} \cdot (1-\varphi)^n $, with $\varphi = \frac{1+\sqrt{5}}{2}$ is the golden mean. We use this to convert the sum into a weighted geometric
series (where each factor is multiplied by a simple polynomial):  \bea \ce(n) & \ = \ &
\frac{\varphi}{\sqrt{5}} \sum_{\ell=0}^{\lfloor \frac{n-3}{2}\rfloor}
(n-3-2\ell) (-1)^\ell \varphi^{n-3-2\ell} + O(F_{n-2}) \nonumber\\
&=& \frac{\varphi^{n-2}}{\sqrt{5}} \left[(n-3) \sum_{\ell=0}^{\lfloor
\frac{n-3}{2}\rfloor} (-\varphi^{-2})^{\ell} -2
\sum_{\ell=0}^{\lfloor \frac{n-3}{2}\rfloor} \ell
(-\varphi^{-2})^{\ell} \right] + O(F_{n-2}).\ \  \ \ \ \ \eea We use the geometric series formula to evaluate the first term, and differentiating identities\footnote{As $\sum_{k=0}^m x^k = (1-x^{m+1})/(1-x)$, applying $x \frac{d}{dx}$ to both sides gives $\sum_{k=0}^m k x^k = x(1 - (m + 1) x^m + m x^{m + 1})/(1 - x)^2$.} on the second. After some algebra we obtain \be \ce(n) \ = \ \frac{n \varphi^{n-2}}{\sqrt{5}(1+\varphi^{-2})} +
O(F_{n-2})\ = \ \frac{n \varphi^n}{\sqrt{5}(\varphi^2+1)} +
O(F_{n-2}). \ee As $F_{n-1} =
\frac{\varphi^n}{\sqrt{5}}  + O(1)$ for large $n$, we finally obtain \be \ce(n) \ =
\ \frac{n F_{n-1}}{\varphi^2+1} + O(F_{n-2}).\ee We note a more careful analysis is possible, and such a computation leads to an exact form for the mean.\\

\item Gaussian Behavior (Theorem \ref{thm:Gaussian}): As we have the density function $P_n(k) = \ncr{n-1-k}{k} / F_{n-1}$, one way to prove the Gaussian behavior is to show that as $n\to\infty$ the function $P_n(k)$ converges to a normal distribution. This may be accomplished via Stirling's formula (see \cite{KKMW} for the computation). Another possible approach would be to generalize the proof of Lekkerkerker's theorem to compute all moments, and then appeal to the Moment Method; however, as the combinatorics become harder as the moment increases, the Stirling approach is more tractable.\\

\end{enumerate}


\section{Positive Linear Recurrence Sequences}\label{sec:plrs}

We discuss the main ideas in proving the Gaussian behavior for the generalized Zeckendorf decompositions for Positive Linear Recurrence Sequences $\{H_n\}$. In particular, we discuss the obstructions that arise in trying to use the argument from the previous section, and describe the techniques that handle them. The computations become quite long and technical; we refer the reader to \cite{MW} for these details, and content ourselves with describing the method below.

The first step is to prove that all integers have a unique, legal representation involving the Positive Linear Recurrence Sequence. The proof is essentially just careful book-keeping. 

The next step is to determine the size of a general term $H_n$ in our sequence. The most important use of this is to count the number of integers we have in a given window, which is used to normalize our counts to probabilities. In the Fibonacci case, the number of integers in the interval $[F_n, F_{n+1})$ is just $F_{n+1} - F_n = F_{n-1}$; we then used Binet's formula to approximate well the size of $F_{n-1}$. This all generalizes immediately as we have a linear recurrence relation.

We first sketch another proof of the results from \S\ref{sec:fibonaccicase}, and then comment on how to generalize these arguments. While we do not discuss the proof of Theorem \ref{thm:lekgaussfardiff} in detail, it too can be handled by this method (see \cite{MW}).

\subsection{New Approach: Case of Fibonacci Numbers}\label{sec:newfibonacci}

We change notation slightly for the rest of this chapter in order to match the notation of \cite{MW}. Let
$p_{n,k} := \#\{N\in [F_n, F_{n+1})$: the Zeckendorf decomposition of $N$ has exactly $k$ summands$\}$. This double sequence satisfies a nice recurrence relation. For any $N\in [F_{n+1}, F_{n+2})$, we have $N=F_{n+1}+F_{t}+\cdots$ where $t\le n-1$ (we are not allowed to have adjacent Fibonacci numbers in our decomposition). Imagine $N$ has exactly $k+1$ summands in its decomposition. It must have $F_{n+1}$ and it cannot have $F_{n}$, and needs exactly $k$ more non-adjacent summands from $F_1$ to $F_{n-1}$. There are $p_{n-1,k}$ ways to have $k$ non-adjacent summands with $F_{n-1}$ included, $p_{n-2,k}$ ways to have $k$ non-adjacent summands without $F_{n-1}$ but with $F_{n-2}$ included, and so on. We thus obtain the following formula for $p_{n+1, k+1}$:
\begin{eqnarray}
p_{n+1,k+1}& \ = \ &p_{n-1,k}+p_{n-2,k}+\cdots. 
\end{eqnarray}
Similarly (replacing $n+1$ with $n$) we find
\begin{eqnarray} 
p_{n,k+1}& \ =  \ &p_{n-2,k}+p_{n-3,k}+\cdots. 
\end{eqnarray}
Subtracting, we find
\begin{eqnarray} p_{n+1,k+1}& \ =  \ &p_{n,k+1}+p_{n-1,k}.
\end{eqnarray}

Our goal is to extract information about the $p_{n,k}$. A powerful approach is to use generating functions. The generating function in this case is\be \sum_{n,k>0}p_{n,k}x^ky^n \ = \ \frac{y}{1-y-xy^2}. \ee Using partial fractions, we find \be \frac{y}{1-y-xy^2} \ = \ -\frac{y}{y_1(x)-y_2(x)}\left(\frac{1}{y-y_1(x)}-\frac{1}{y-y_2(x)}\right),\ee
where $y_1(x)$ and $y_2(x)$ are the roots of $1-y-xy^2=0$ and the coefficient of $y^n$ is $g(x)=\sum_{k>0}p_{n,k}x^k$. Note the roots are readily computed via the quadratic formula; this is not true for the general case, and is the source of much of the technicalities.

As in the introduction, let $K_n$ be the corresponding random variable associated with $k$. Using the Method of Moments, it suffices to prove the moments of $K_n$ converge to the moments of a Gaussian to prove Gaussian behavior. We do so through differentiating identities. With $g(x)$ as above, differentiating once and setting $x=1$ yields \be g(1) \ =  \ \sum_{k>0}p_{n,k} \ = \ F_{n+1}-F_n, \ee which is just the number of elements in our interval $[F_n, F_{n+1})$. Differentiating again gives\be g'(x) \ = \ \sum_{k>0}kp_{n,k}x^{k-1}.\ee If we take $x=1$ we essentially obtain the mean of $K_n$; we need to divide by $F_{n-1}$ as the $p_{n,k}$'s are counts and not probabilities. As $g(1) = F_{n-1}$, we find  \be g'(1) \ =  \ g(1)E[K_n].\ee We continue, and find \be \left(xg'(x)\right)' \ = \ \sum_{k>0}k^2p_{n,k}x^{k-1},\ee which leads to \be \left(xg'(x)\right)'\Big|_{x=1}\ = \ g(1)E[K_n^2], \ee and then \be \left(x\left(xg'(x)\right)'\right)'\Big|_{x=1} \ = \ g(1)E[K_n^3],\ee  and so on.

Similar results hold for the centralized random variable $K'_n=K_n-E[K_n]$. Miller and Wang \cite{MW} prove that $E[(K'_n)^{2m}]/({\rm SD}(K'_n))^{2m}\rightarrow (2m-1)!!$ (with ${\rm SD}(K_n')$ the standard deviation of $K_n'$) and $E[(K'_n)^{2m-1}]/({\rm SD}(K'_n))^{2m-1}$ $\rightarrow$ $0$, which yields the Gaussian behavior of $K_n$.

\subsection{New Approach: General Case}

We generalize the arguments from \S\ref{sec:newfibonacci} to the case of Zeckendorf expansions arising from Positive Linear Recurrence Sequences $\{H_n\}$. We now set $p_{n,k}= \#\{N\in [H_n, H_{n+1})$: the generalized Zeckendorf decomposition of $N$ has exactly $k$ summands$\}$. We can find a recurrence relation as before. For the Fibonacci numbers, we found \be p_{n+1,k+1} \ = \ p_{n,k+1}+p_{n,k}.\ee In the general case, we have\be p_{n+1,k}\ = \ \sum_{m=0}^{L-1}\sum_{j=s_m}^{s_{m+1}-1}p_{n-m,k-j},\ee where $s_0=0, s_m=c_1+c_2+\cdots+c_m$ (with these $c$'s the same $c$'s as in the definition of our recurrence sequence).

The generating function in the Fibonacci case was relatively straightforward, being $\frac{y}{1-y-xy^2}$. Now our generating function equals \begin{equation}
\frac{\sum_{n\le L}p_{n,k}x^k y^n-\sum_{m=0}^{L-1}\sum_{j=s_m}^{s_{m+1}-1} x^j y^{m+1} \sum_{n<L-m}p_{n,k}x^k y^n}{1-\sum_{m=0}^{L-1}\sum_{j=s_m}^{s_{m+1}-1}x^j y^{m+1}}.
\end{equation}

We again perform a partial fraction expansion in order to glean information about the coefficients. Instead of \be -\frac{y}{y_1(x)-y_2(x)}\left(\frac{1}{y-y_1(x)}-\frac{1}{y-y_2(x)}\right), \ee now we have \be-\frac{1}{\sum_{j=s_{L-1}}^{s_L-1}x^j} \sum_{i=1}^{L}\frac{B(x,y)}{(y-y_i(x))\prod_{j\neq i}\left(y_j(x)-y_i(x)\right)}, \ee where\be B(x,y) \ = \ \sum_{n\le L}p_{n,k}x^k y^n-\sum_{m=0}^{L-1}\sum_{j=s_m}^{s_{m+1}-1} x^j y^{m+1} \sum_{n<L-m}p_{n,k}x^k y^n \ee and the $y_i(x)$'s are the roots of $A(y)=1-\sum_{m=0}^{L-1}\sum_{j=s_m}^{s_{m+1}-1}x^j y^{m+1}=0$. One of the major difficulties of the proof is the analysis of the roots $y_i(x)$. Unlike the Fibonacci case, where it was easy to write down simple expressions for these via the quadratic formula, in general the arguments become quite involved. If however the coefficients in the recurrence relation are non-increasing then the proofs are easy; see Appendix C of \cite{MW}. The general cases (see Appendix A of \cite{MW}) is more complicated, involving continuity and the range of the $|y_i(x)|$'s. The main idea is to first show that there exists $x > 0$ such that $A(y)$
has no multiple roots and then prove that there are only finitely many $x > 0$ such that $A(y)$ has multiple roots.

The coefficient of $y^n$ is $g(x)=\sum_{k>0}p_{n,k}x^k$. We use the method of differentiating identities as before to get the moments, and find that $K_n$ converges to a Gaussian as $n\to\infty$.

\ \\

The proof of Theorem \ref{thm:lekgaussfardiff} proceeds similarly. We state the two key elements, the recurrence relation and the generating function, without proof, and refer the reader to \cite{MW} for the details. Let $p_{n,k,\ell}$ be the number of far-difference representations of integers in $(S_{n-1},S_n]$ with $k$ positive summands and $\ell$ negative summands. The recurrence relation is
 \begin{equation}
p_{n,k,\ell} \ = \ p_{n-1,k,\ell}+p_{n-4,k-1,\ell}+p_{n-3,\ell,k-1},\ n\ge 5,
\end{equation} and the generating function is \be \hat{\mathscr{G}}(x,y,z) \ = \ \sum_{n>0,k>0,\ell\ge 0}p_{n,k,\ell}x^ky^\ell z^n. \ee


\section{Conclusion and Future Research}\label{sec:conclusion}

The combinatorial approach has extended previous work, allowing us to prove Gaussian behavior for the number of summands for a large class of expansions in terms of solutions to linear recurrence relations. This is just the first of many questions one can ask. Others, which we hope to return to at a later date, include:
\begin{enumerate}
\item Are there similar results for linearly recursive sequences with arbitrary integer coefficients (i.e., negative coefficients are allowed in the defining relation)?\\

\item Lekkerkerker's theorem, and the Gaussian extension, are for the behavior in intervals $[F_n, F_{n+1})$.
Do the limits exist if we consider other intervals, say
$[F_n+g_1(F_n), F_n + g_2(F_n))$ for some functions $g_1$ and $g_2$? If yes, what must be
true about the growth rates of $g_1$ and $g_2$?\\

\item For the generalized recurrence relations, what happens if instead of looking at $\sum_{i=1}^n a_i$ we study $\sum_{i=1}^n \min(1,a_i)$? In other words, we only care about how many distinct $H_i$'s occur in the decomposition.\\

\item What can we say about the distribution of the largest gap between summands in the Zeckendorf decomposition? Appropriately normalized, how does the distribution of gaps between the summands behave? What is the distribution of the largest gap? How often is there a gap of 2? 
\end{enumerate}


$ $

\end{document}